\documentclass[a4paper, 11pt]{amsart}
\usepackage{amsfonts, amsmath, amssymb, amsthm,mathtools, url,dsfont}
\usepackage{geometry}
\usepackage{graphicx}
\usepackage[english] {babel}
\usepackage{enumerate}
\usepackage{array}
\usepackage{color}
\usepackage[breaklinks]{hyperref}

\newtheorem{theorem}{Theorem}[section]
\newtheorem{cor}[theorem]{Corollary}
\newtheorem{lemma}[theorem]{Lemma}

\newcommand{\p}{\mathbb{P}}
\renewcommand{\Re}{\operatorname{Re}}
\newcommand{\meas}{\operatorname{meas}}

\begin{document}

\begin{abstract}
We prove that any non-zero complex values $z_1,\ldots,z_n$ can be approximated by the following integral shifts of the Riemann zeta-function $\zeta(s+id_1\tau),\ldots,\zeta(s+id_n\tau)$ for infinitely many $\tau$, provided $d_1,\ldots,d_n\in\mathbb{N}$ and $s$ is a fixed complex number lying in the right open half of the critical strip.
\end{abstract}

\title{Joint value-distribution of shifts of the Riemann zeta-function}

\author{\L ukasz Pa\'nkowski} 
\address{Faculty of Mathematics and Computer Science, Adam Mickiewicz University, Uniwersytetu Pozna\'nskiego 4, 61-614 Pozna\'n, Poland}
\email{lpan@amu.edu.pl}

\thanks{The author was partially supported by the grant no. 2016/23/D/ST1/01149 from the National Science Centre.}

\maketitle

\section{Introduction}

In the 80's of the last century Bagchi (see \cite{B81,B87}) observed that the classical Riemann Hypothesis is equivalent to the fact that for  every $\varepsilon>0$ and every compact set $K\subset D:=\{s\in\mathbb{C}:1/2<\Re(s)<1\}$ with connected complement we have

\begin{equation}\label{eq:SelfSimZeta}
\liminf_{T\to\infty}\frac{1}{T}\meas\left\{\tau\in[0,T]:\max_{s\in K}|\zeta(s+i\tau) - \zeta(s)|<\varepsilon\right\}>0,
\end{equation}
where $\meas$ denotes the Lebesgue measure on $\mathbb{R}$. In other words \eqref{eq:SelfSimZeta} means that the set of real~$\tau$ satisfying $\max_{s\in K}|\zeta(s+i\tau) - \zeta(s)|<\varepsilon$ has a positive lower density. Moreover, let us note that in the language of topological dynamics Inequality (\ref{eq:SelfSimZeta}) says that the Riemann Hypothesis is equivalent to the strong recurrence of the Riemann zeta-function (see~\cite{GottschalkHedlund}). 

One implication in Bagchi's observation is an immediate consequence of the so-called universality theorem due to Voronin \cite{V}, which is a generalization of the work of Bohr and his collaborators (see \cite{Bohr, BohrCourant,BohrJessen}) on denseness theorems in $\mathbb{C}$ of values of the Riemann zeta-function and  states that for any non-vanishing and continuous function $f(s)$ on $K$, analytic in the interior of $K$, and every $\varepsilon>0$, we have
\[
\liminf_{T\to\infty}\frac{1}{T}\meas\left\{\tau\in[0,T]:\max_{s\in K}|\zeta(s+i\tau) - f(s)|<\varepsilon\right\}>0.
\]
The converse implication is a combination of Rouch\'e's theorem and the well-known zero density theorem for the Riemann zeta-function.  

An interesting generalization of \eqref{eq:SelfSimZeta} was suggested by Nakamura in \cite{N09}, where he introduced the problem of finding all $d_1,d_2,\ldots,d_n$ such that for every compact set $K\subset D$ with connected complement and every $\varepsilon>0$ we have
\begin{equation}\label{eq:Self}
\liminf_{T\to\infty}\frac{1}{T}\meas\left\{\tau\in[0,T]:\max_{1\leq k\ne\ell\leq n}\max_{s\in K}|\zeta(s+id_k\tau) - \zeta(s+id_\ell\tau)|<\varepsilon\right\}>0.
\end{equation}
Notice that \eqref{eq:Self} holds for $n=2$ and $d_1=0$, $d_2\ne 0$ if and only if the Riemann Hypothesis is true. Interestingly it turns out (see~\cite{N09,NP,PanWuerzburg,P2019}) that for $n=2$ the case $d_1=0$, $d_2\ne 0$ is the most difficult one as it was shown that \eqref{eq:Self} holds for $n=2$ and every non-zero real $d_1,d_2$. A common strategy of showing \eqref{eq:Self} is to prove a joint version of universality for $\zeta(s+id_1\tau),\ldots,\zeta(s+id_n\tau)$, namely to prove that for any non-vanishing and continuous functions $f_1(s),\ldots,f_n(s)$ on $K$, analytic in the interior of $K$, and every $\varepsilon>0$ we have
\begin{equation}\label{eq:jointSelf}
\liminf_{T\to\infty}\frac{1}{T}\meas\left\{\tau\in[0,T]:\max_{1\leq k\leq n}\max_{s\in K}|\zeta(s+id_k\tau) - f_k(s)|<\varepsilon\right\}>0.
\end{equation}
Obviously, the truth of \eqref{eq:jointSelf} implies \eqref{eq:Self}. The problem of finding all parameters $d_1,d_2$ satisfying \eqref{eq:jointSelf} for $n=2$ was also completely solved in the aforementioned series of papers by Nakamura and the author, where it was proved that \eqref{eq:jointSelf} holds for $n=2$ and any real $d_1,d_2$ with $0\ne d_1,d_2$ and $d_1\ne \pm d_2$. Note that it is indeed a complete answer for $n=2$ as one cannot expect that \eqref{eq:jointSelf} holds for $d_1=\pm d_2$.

Unfortunately, the situation for $n\geq 3$ is much more complicated and so far we know only (see \cite{N09}) that \eqref{eq:jointSelf} holds if $d_1=1,d_2,\ldots,d_n$ are algebraic real numbers linearly independent over $\mathbb{Q}$, which is an easier case since for such $d_k$'s, in view of Baker's theorem in transcendental number theory, the set $I(d_1,\ldots,d_n):=\{d_k\log p: p\text{ is prime, }1\leq k\leq n\}$ is linearly independent over $\mathbb{Q}$ and one can follow straightforwardly the proof of Voronin's theorem in order to get \eqref{eq:jointSelf}. However, if the set $I(d_1,\ldots,d_n)$ is not linearly independent over $\mathbb{Q}$, we need a new approach or at least a significant modification of known proofs of universality, since linear independence of the set $I(d_1,\ldots,d_n)$ is crucial in several steps of these proofs. 

In the paper we focus on considering the following weaker version of \eqref{eq:jointSelf} in the flavor of denseness theorems due to Bohr, where $K$ is a singleton and $d_1,\ldots,d_n$ are positive integers. Note that in this case the set $I(d_1,\ldots,d_n)$ is far from being linearly independent over the field of rational numbers.
\begin{theorem}\label{thm:main}
	Let $d_1<d_2<\cdots<d_n\in\mathbb{N}$ and $s=\sigma+it$ with $\sigma\in(1/2,1]$ and $t>0$.  Then, for every $z_1,z_2,\ldots,z_n\in\mathbb{C}$ and every $\varepsilon>0$ the set of all positive real $\tau$ satisfying 
	\[
	\max_{1\leq k\leq n}|\log\zeta(s+id_k\tau)-z_k|<\varepsilon
	\]
	has a positive lower density.
\end{theorem}
Obviously, the theorem easily implies that \eqref{eq:jointSelf} holds for $K=\{s\}$ and $f_k(s)=z_k\ne 0$. 
Moreover, it is worth mentioning that the proof of Theorem \ref{thm:main} presented in the paper might be simplified in several places under the assumption that $d_1,\ldots,d_n$ are integers. However we decided to make all steps in the proof as general as possible since it might be interesting for the reader to see that in fact the proof of the above theorem is valid for arbitrary positive real $d_1,\ldots,d_n$ except only one step where we use the implication 
\[
-\tau\frac{\log p}{2\pi}\text{ is close to }\theta_p \Longrightarrow -\tau d_k\frac{\log p}{2\pi}\text{ is close to }d_k\theta_p,
\]
which is trivially true for $d_k\in\mathbb{Z}$, but it does not necessary hold if $d_k\notin\mathbb{Z}$.

\section{Denseness lemma}

In this section we prove the following denseness lemma, which holds for all positive real $d_1,\ldots,d_n$.

\begin{lemma}\label{lem:denseness}
	Let $d_1<d_2<\cdots<d_n\in\mathbb{R}_+$ and $s=\sigma+it$ with $\sigma\in(1/2,1]$ and $t>0$.  Then, for every $z_1,z_2,\ldots,z_n\in\mathbb{C}$ and every $\varepsilon>0$, $y>0$, there exist a finite set of prime numbers $M\supset\{p: p<y\}$ and a sequence $\overline{\theta} = (\theta_p)_{p\in M}$ such that
	\[
	\max_{1\leq k\leq n}\left|\sum_{p\in M}\frac{e(d_k\theta_p)}{p^s} - z_k\right|<\varepsilon,
	\]
	where as usual $e(t)=e^{2\pi i t}$.
\end{lemma}

The proof of the lemma essentially based on the following Pechersky's \cite{Pe} generalization of the Riemann rearrangement lemma to Hilbert spaces.

\begin{lemma}[{\cite{Pe}}]\label{lem:Pech}
Let $H$ be a real Hilbert space with the inner product $(\cdot|\cdot)$ and a sequence $(u_n)_{n\in\mathbb{N}}\subset H$ be such that $\sum_{n=1}^\infty\Vert u_n\Vert_H^2<\infty$. Moreover, suppose that for every $e\in H$ with $\Vert e\Vert =1$ the series $\sum_{n=1}^\infty \langle u_n,e\rangle$ are conditionally convergent after suitable permutation of terms. Then, for every $v\in H$ there exists a permutation $(n_k)$ such that $\sum_{k=1}^\infty u_{n_k} = v$.
\end{lemma}

Moreover, we use the following effective version of Kronecker approximation theorem due to Chen. Henceforth, $\|\cdot \|$ will denote the distance to the nearest integer. 
\begin{lemma}[Chen, \cite{chen}]  \label{lem:chen}
	Let $\lambda_1, \hdots, \lambda_n$ and $\alpha_1, \hdots, \alpha_n$ be real numbers and assume the following property:  
	for all integers $u_1, \hdots, u_n$ with $|u_j|\le M$, the assertion 
	\[
	u_1\lambda_1+ \cdots + u_n\lambda_n \in\mathbb{Z}
	\]
	implies that $u_1\alpha_1+ \cdots + u_n\alpha_n$ is an integer. Then for all positive real numbers $\delta_1, \hdots, \delta_n$ and for any integers $T_1<T_2$, we have 
	\[
	\inf_{t \in [T_1, T_2]\cap\mathbb{Z}} \sum_{j=1}^n \delta_j \| \lambda_j t -\alpha_j\|^2 \le \frac{ \Delta}{4} \sin^2\left(\frac{\pi}{2(M+1)}\right) + \frac{\Delta M^n}{8 (T_2-T_1)\Lambda  }, 
	\] 
	where 
	\[
	\Delta = \sum_{j=1}^n \delta_j 
	\]
	and 
	\[
	\Lambda = \min\left\lbrace  \left| u_1\lambda_1+ \cdots + u_n\lambda_n  \right|  : u_j \in \mathbb{Z}, |u_j| \le M, \sum_{j=1}^n \lambda_ju_j \not\in\mathbb{Z}
	\right\rbrace .
	\]
\end{lemma}

\begin{cor}\label{cor:fromChen}
Let $\omega\geq 1$, $n\in\mathbb{N}$ and $d_1,d_2,\ldots,d_n\in\mathbb{R}$. Then there is an integer $T:=T(d_1,d_2,\ldots,d_n,n,\omega)>0$ such that for every $a\in\mathbb{Z}$ there exists an integer $h\in[a,a+T]$ satisfying
\[
\sum_{k=1}^n\|hd_k\|^2\leq \frac{n}{\omega}.
\]
\end{cor}
\begin{proof}
Applying Chen's lemma for $T_1=a$, $T_2=a+T$, $(\lambda_1,\ldots,\lambda_n)=(d_1,d_2,\ldots,d_n)$, $\alpha_1=\ldots=\alpha_n=0$, $\delta_1=\ldots=\delta_n=1$ and $M:=\pi\omega/4-1$ gives that there is $h\in\mathbb{Z}$, $h\in[a,a+T]$ such that
\[
\sum_{k=1}^n\|hd_k\|^2\leq \frac{n}{4}\left(\sin^2\left(\frac{\pi}{2(M+1)}\right)+\frac{M^n}{2\Lambda T}\right).
\]
Obviously, we have $\sin^2\left(\frac{\pi}{2(M+1)}\right)\leq \frac{2}{\omega}$. Moreover, since $M$ depends only on $\omega$ and $\Lambda$ depends only on $M$ and $d_1,\ldots,d_n$, there is a sufficiently large $T$ such that $\frac{M^n}{2\Lambda T}\leq \frac{2}{\omega}$, which completes the proof.
\end{proof}

The need for the application of Lemma \ref{lem:Pech} forces us to consider the distribution of zeros of functions being linear combinations of exponential functions, which is well understood by the following result due to Wilder.
\begin{lemma}[Wilder's theorem]
Let $g(z) = \sum_{k=1}^n A_ke^{\omega_kz}$, where $z=x+iy\in\mathbb{C}$, $\mathbb{C}\ni A_k\ne 0$, $k=1,2,\ldots,n$, and $\omega_1<\ldots\omega_n$. Then there exsits $K>0$ such that
\begin{enumerate}[\quad\upshape (i)]
	\item all zeros of $g(z)$ lie in the strip $|x|\leq K$;
	\item for each pair $(\alpha,\beta)\in\mathbb{R}^2$ with $\beta>0$, we have
	\[
	\left|N_0(\alpha,\beta,K) - \beta\frac{\omega_n-\omega_1}{2\pi}\right|\leq n-1,
	\]
	where $N_0(\alpha,\beta,K)$ counts the number of zeros of $g(z)$ in the rectangle $|x|\leq K$, $y\in[\alpha,\alpha+\beta]$.
\end{enumerate}
\end{lemma}
\begin{proof}
	For the proof we refer to \cite[Theorem 1]{Dickson}.
\end{proof}

\begin{proof}[Proof of Lemma \ref{lem:denseness}]
Let $\theta_{p_m} = \frac{m}{L}$, where $p_m$ denotes the $m$-th prime number and $L$ is an arbitrary integer satisfying $L>n-1+d_n-d_1$. We shall use Lemma \ref{lem:Pech} for the real Hilbert space $H=\mathbb{C}^n$ equipped with the standard inner product and the sequence
\[
u_m = \left(\frac{e(d_1\theta_{p_m})}{p_m^s},\ldots,\frac{e(d_n\theta_{p_m})}{p_m^s}\right).
\]
Obviously, $\sum_{m=1}^\infty\Vert u_m\Vert^2<\infty$, since $\sigma>1/2$. Thus, in order to apply Lemma \ref{lem:Pech}, it suffices to show that for every $e=(a_1,\ldots,a_n)\in \mathbb{C}^n$ with $\Vert e\Vert_H =1$ there are two permutations of  the series $\sum_p (u_m|e)$ tending to $+\infty$ and $-\infty$, respectively. In fact, we show only the existence of a permutation of the series, which diverges to $+\infty$, since the case $-\infty$ is similar and can be left to the reader.

Let us note that
\[
(u_m|e) = \Re\sum_{k=1}^n \overline{a_k}\frac{e(d_k\theta_{p_m})}{p_m^s} = \frac{1}{p_m^\sigma}\Re\left(e^{-it\log p_m}\left(\sum_{k=1}^n\overline{a_k}e(d_k\theta_{p_m})\right)\right)
\]
and consider the function $f(x) = \sum_{k=1}^n\overline{a_k}e(d_k x)$, $x\in\mathbb{R}$. Putting $A=\sum_{k=1}^n |a_k|^2$, one can easily deduce from the Cauchy-Schwarz inequality that for every $x\in\mathbb{R}$ we have
\begin{align*}
|f(x+\tau)-f(x)| = \left|\sum_{k=1}^n\overline{a_k}e(d_k x)(e(d_k\tau) - 1)\right|\leq 2\pi \sum_{k=1}^n|a_k|\Vert d_k\tau\Vert\leq 2\pi\sqrt{A}\left(\sum_{k=1}^n\Vert d_k\tau\Vert^2\right)^{1/2}.
\end{align*}
Now, from Corollary \ref{cor:fromChen}, we see that there is a positive integer $N$ such that every interval of length $N$ contains an integer $h$ satisfying $\sum_{k=1}^n\|hd_k\|^2\leq \frac{\varepsilon^2}{(2\pi)^2A}$. Therefore, dividing the interval $(0,\infty)$ into disjoint intervals $I_l$, $l\in\mathbb{N}$, of length $N$, one can find integers $\tau_l\in I_l$ satisfying $\sum_{k=1}^n\|\tau_ld_k\|^2\leq \frac{\varepsilon^2}{(2\pi)^2A}$. Hence, for every $x\in\mathbb{R}$ and $l\in\mathbb{N}$, we have
\[
|f(x+\tau_l)-f(x)|\leq \varepsilon.
\]

Now, let us observe that the function $f(x)$ has at most $n-1+d_n-d_1$ zeros in the interval $[0,1]$. Hence, by the choice of $L$, there is $m_0$ such that $0\leq m_0\leq L-1$ and $f(\frac{m_0}{L})=c_0\ne 0$. Thus for every $\tau_l$ we have
\[
f\left(\frac{m_0+L\tau_l}{L}\right) = f\left(\frac{m_0}{L}+\tau_l\right) = c_0+\xi_l,\qquad\text{for suitable $|\xi_l|\leq \varepsilon$},
\]
so each interval $[m_0+kLN,m_0+(k+1)NL]$, $k=1,2,\ldots$, contains an integer $m_k'$ satisfying $f(\frac{m_k'}{L})=c_0+\xi_l$ with $|\xi_l|\leq \varepsilon$. Assume that $c_0 = |c_0|e^{i\varphi}$. Then for every $m_k'$ and $m\in\mathbb{N}$ such that
\[
\frac{2\pi m}{t}-\frac{\pi}{4t}+\frac{\varphi}{t} < \log p_{m_k'}<\frac{2\pi m}{t}+\frac{\pi}{4t}+\frac{\varphi}{t}
\]
and sufficiently small $\varepsilon>0$ we have
\[
\Re\left(e^{-it\log p_{m_k'}}f\left(\frac{m_k'}{L}\right)\right)\geq c'>0.
\]
Hence it remains to notice that simple calculations and the prime number theorem imply that
\[
\sum_{\substack{m_k'\\\frac{2\pi m}{t}-\frac{\pi}{4t}+\frac{\varphi}{t}<\log p_{m_k'}<\frac{2\pi m}{t}+\frac{\pi}{4t}+\frac{\varphi}{t}}}\frac{1}{p_{m_k'}^\sigma}\gg \frac{1}{m},
\]
so using the fact that the harmonic series is divergent completes the proof.
\end{proof}

\section{Approximation by a finite sum over primes}

In this section we use the following Tsang's lemma in order to show that the logarithm of the Riemann zeta-function can be approximated by a truncated Dirichlet series over primes for sufficiently many arguments from any vertical line lying in the right open half of the critical strip.
\begin{lemma}[Tsang \cite{Tsang}]
Let $\sigma\in(1/2,1]$ and $0<\delta<\frac{1}{64}$ be fixed. Then, for sufficiently large $T$, we have
\[
\frac{1}{T}\int_T^{2T}\left|\log\zeta(\sigma+it) - \sum_{p\leq T^\delta}\frac{1}{p^{s+it}}\right|^2\ll T^{c(1/2-\sigma)},
\]
where $c$ is a positive constant.
\end{lemma}

\begin{lemma}\label{lem:approxFiniteSum}
	Let $\varepsilon,\varepsilon'>0$, $d_1<d_2<\ldots<d_n$, $s=\sigma+it$ with $\sigma\in(1/2,1], t>0$, $0<\delta<\frac{1}{65}$ and $C(T,\varepsilon)$ denotes the set of $\tau\in[T,2T]$ satisfying
	\[
	\max_{1\leq k\leq n}\left|\log\zeta(\sigma+id_k\tau) - \sum_{p\leq T^\delta}\frac{1}{p^{s+id_k\tau}}\right|<\varepsilon.
	\]
	Then, for sufficiently large $T$, we have
	\[
	\frac{1}{T}\meas C(T,\varepsilon)>1-\varepsilon'.
	\]
\end{lemma}
\begin{proof}
First let us observe that using Tsang's lemma and a suitable version of Hilbert's inequality due to Montgomery and Vaughan (see \cite{MV}) give
\begin{align*}
\int_{T}^{2T}\left|\log\zeta(s+id_k\tau)-\sum_{p\leq T^\delta}\frac{1}{p^{s+id_k\tau}} \right|^2d\tau 
&\ll \int_{d_kT+t}^{2(d_kT+t)}\left|\log\zeta(\sigma+i\tau)-\sum_{p\leq T^\delta}\frac{1}{p^{\sigma+i\tau}} \right|^2d\tau\\
&\ll \int_{d_kT+t}^{2(d_kT+t)}\left|\log\zeta(\sigma+i\tau)-\sum_{p\leq (d_kT+t)^{1/65}}\frac{1}{p^{\sigma+i\tau}} \right|^2d\tau\\
&\quad+\int_{d_kT+t}^{2(d_kT+t)}\left|\sum_{T^\delta\leq p\leq(d_kT+t)^{1/65}}\frac{1}{p^{\sigma+i\tau}} \right|^2d\tau\\
&\ll T^{1+c(1/2-\sigma)} + T\sum_{p\leq (d_jT+t)^{1/65}}\frac{1}{p^{2\sigma}}\\
&\ll T^{1+c'(1/2-\sigma)}
\end{align*}
for some $c'>0$ and $k=1,2,\ldots,n$. Hence for every $\varepsilon_0>0$ and sufficiently large $T$ we have
\[
\sum_{k=1}^n\frac{1}{T}\int_{T}^{2T}\left|\log\zeta(s+id_k\tau)-\sum_{p\leq T^\delta}\frac{1}{p^{s+id_k\tau}} \right|^2d\tau <\varepsilon_0.
\]

Now let 
\begin{align*}
A_T' &= \left\{\tau\in[T,2T]: \sum_{k=1}^n|\log\zeta(s+id_k\tau) - \sum_{p\leq T^\delta}\frac{1}{p^{s+id_k\tau}}|>\varepsilon_0^{2/3}\right\},\\
A_T'' &= \left\{\tau\in[T,2T]: \exists_{1\leq k\leq n}\zeta(s+id_k\tau)=0\right\}.
\end{align*}
It is well known that $A_T''\ll T^{1+c''(1/2-\sigma)}\log T$ for some $c''>0$, so for suffciently large $T$ we have $A_T''\ll \varepsilon_0 T$. Moreover,
\[
\varepsilon_0^{2/3}\frac{\meas(A_T')}{T}<\sum_{k=1}^n\frac{1}{T}\int_{A_T'}\left|\log\zeta(s+id_k\tau)-\sum_{p\leq T^\delta}\frac{1}{p^{s+id_k\tau}} \right|^2d\tau\leq \varepsilon_0,
\]
so $\meas(A_T')\leq \varepsilon_0^{1/3}T$. Thus, putting $A_T=A_T'\cup A_T''$ and choosing $\varepsilon_0$ sufficiently small, we get that  $\meas(A_T)\leq \varepsilon'$ and 
\[
\sum_{k=1}^n|\log\zeta(s+id_k\tau) - \sum_{p\leq T^\delta}\frac{1}{p^{s+id_k\tau}}|<\varepsilon
\]
for $\tau\in [T,2T]\setminus A_T$.
\end{proof}

\section{Proof of the main theorem}

In order to prove our main theorem we need the following crucial lemma which allows us to connect Lemma \ref{lem:denseness} and Lemma \ref{lem:approxFiniteSum}.
\begin{lemma}\label{lem:middle}
Let $d_1<d_2<\cdots<d_n\in\mathbb{R}_+$ and $s=\sigma+it$ with $\sigma\in(1/2,1]$ and $t>0$. Moreover, suppose that $y>0$, $0<\delta<\frac{1}{65}$, $0<d<1/2$, $M$ is a finite set of prime numbers such that $\{p:p\leq y\}\subset M$ and $(\theta_p)_{p\in M}$ is a sequence of real numbers. Then for sufficiently large $T$ we have
\[
\max_{1\leq k\leq n}\int_{A_d(T)}\left|\sum_{\substack{p\leq T^\delta\\p\not\in M}}\frac{1}{p^{s+id_k\tau}}\right|^2d\tau\ll T(2d)^{|M|}y^{1-2\sigma},
\]
where $A_d(T) = \{\tau\in [T,2T]: \max_{p\in M}\|-\tau\frac{\log p}{2\pi}-\theta_p\|\leq d\}$.
\end{lemma}
\begin{proof}
	Let $\delta'>0$ be such that $0<d+\delta'<1/2$ and $\xi:\mathbb{R}\to [0,1]$ be a continuous and periodic function with period $1$ defined on $[-1/2,1/2]$ by
	\[
	\xi(x) = \begin{cases}
	1,&|x|\leq d,\\
	-\frac{|x|}{\delta'}+\frac{d+\delta'}{\delta'},&d<|x|\leq d+\delta',\\
	0,&d+\delta'<|x|\leq \frac{1}{2}.
	\end{cases}
	\]
Moreover, we put 
\[
\tilde{\xi}(\tau) = \prod_{p\in M}\xi\left(-\frac{\tau\log p}{2\pi}- \theta_p\right)\qquad\text{for $\tau\in\mathbb{R}$}.
\]
Then $\xi(x) = 1$ if $\|x\|\leq d$ and for $\tau\in A_d(T)$ we have $1=\tilde{\xi}(\tau) = \tilde{\xi}^2(\tau)$, and for $\tau\in\mathbb{R}$ we have $0\leq \tilde{\xi}^2(\tau)\leq \tilde{\xi}(\tau)\leq 1$.

Now let us assume that $\varepsilon_0>0$ is small. Then there is a trigonometric polynomial
\[
P(x) = \sum_{-L\leq \ell\leq L}c(\ell)e(\ell x),\qquad c(\ell)\in\mathbb{C}
\]
such that $|\xi(x)-P(x)|<\varepsilon_0$ for $\tau\in\mathbb{R}$. So let us define 
\[
\tilde{P}(\tau) = \prod_{p\in M}P\left(-\frac{\tau\log p}{2\pi}- \theta_p\right).
\]
Using the fact that $\left|\prod_{j=1}^{n}a_j-\prod_{j=1}^{n}b_j\right|\leq R^{n-1}\sum_{j=1}^n|a_j-b_j|$ for arbitrary complex numbers $a_1,\ldots,a_n,b_1,\ldots,b_n$ with $|a_j|,|b_j|\leq R$, $j=1,2,\ldots,n$, we get
\begin{align*}
\left|\prod_{p\in M}\xi(x_p)-\prod_{p\in M}P(x_p)\right| = \left|\prod_{p\in M}\xi(x_p)-\prod_{p\in M}\xi(x_p)+(P(x_p)-\xi(x_p))\right|\leq 2^{|M|-1}|M|\varepsilon_0
\end{align*}
for all real $x_p$. Hence $|\tilde{\xi}(\tau)-\tilde{P}(\tau)|\leq 2^{|M|-1}|M|\varepsilon_0$ for $\tau\in\mathbb{R}$. Therefore, by the Cauchy--Schwarz inequality, we have
\[
|\tilde{\xi}(\tau)|^2 \leq 2|\tilde{P}(\tau)|^2+2^{2|M|-1}|M|^2\varepsilon_0^2,\qquad\tau\in\mathbb{R}
\]
and
\[
|\tilde{P}(\tau)|^2 \leq 2|\tilde{\xi}(\tau)|^2+2^{2|M|-1}|M|^2\varepsilon_0^2,\qquad\tau\in\mathbb{R}.
\]

Now, let $N'>y$. Then for $\tau\in\mathbb{R}$ we have
\[
\sum_{\substack{p\leq T^\delta\\p\not\in M}}\frac{1}{p^{s+id_k\tau}} = \sum_{\substack{p\leq N'\\p\not\in M}}\frac{1}{p^{s}}e\left(-\frac{d_k\tau\log p}{2\pi}\right)+\sum_{\substack{N'<p\leq T^\delta\\p\not\in M}}\frac{1}{p^{s}}e\left(-\frac{d_k\tau\log p}{2\pi}\right)=: S_1+S_2,
\]
so using the Cauchy--Schwarz inequality again yields
\[
\int_{A_d(T)}\left|\sum_{\substack{p\leq T^\delta\\p\not\in M}}\frac{1}{p^{s+id_k\tau}}\right|^2d\tau\leq 2\int_{A_d(T)}|S_1|^2d\tau+2\int_{A_d(T)}|S_2|^2d\tau =: I_1 + I_2.
\]

First let us estimate $I_1$. We have
\begin{align*}
I_1 %= \int_{A_d(T)}|\tilde{\xi}(\tau)|^2|S_1|^2d\tau
\leq \int_T^{2T}|\tilde{\xi}(\tau)|^2|S_1|^2d\tau\ll\int_T^{2T}|\tilde{P}(\tau)|^2|S_1|^2d\tau + 2^{2|M|-1}|M|^2\varepsilon_0^2\int_T^{2T}|S_1|^2d\tau
\end{align*}
and
\begin{align}
\tilde{P}(\tau) &= \prod_{p\in M}\left(\sum_{-L\leq \ell_p\leq L}c(\ell_p)e\left(\ell_p\left(-\frac{\tau\log  p}{2\pi}-\theta_p\right)\right)\right)\nonumber\\
&=\sum_{-L\leq \ell_{p_1},\ldots,\ell_{p_{|M|}}\leq L}\tilde{c}(\ell_{p_1},\ldots,\ell_{p_{|M|}})e\left(-\frac{\tau}{2\pi}\sum_{i=1}^{|M|}\ell_{p_j}\log p_j\right)\label{eq:Ptilde}
\end{align}
where $p_1,\ldots,p_{|M|}$ are all elements of $M$ and
\[
\tilde{c}(\ell_{p_1},\ldots,\ell_{p_{|M|}}) = c(\ell_{p_1})\cdots c(\ell_{p_{|M|}}) e\left(-\sum_{i=1}^{|M|}\ell_{p_j}\theta_{p_j}\right).
\]
Then
\begin{multline*}
\int_T^{2T}|\tilde{P}(\tau)|^2|S_1|^2d\tau \\ = \int_T^{2T}\left|\sum_{-L\leq \ell_{p_1},\ldots,\ell_{p_{|M|}}\leq L}\sum_{\substack{p\leq N'\\p\not\in M}}\tilde{c}(\ell_{p_1},\ldots,\ell_{p_{|M|}})\frac{1}{p^s}e\left(-\frac{\tau}{2\pi}\left(d_k\log p+\sum_{j=1}^{|M|}\ell_{p_j}\log p_j \right)\right)\right|^2 d\tau
\end{multline*}
Let \[
\Omega_k = \left\{-\left(d_k\log p+\sum_{j=1}^{|M|}\ell_{p_j}\log p_j \right):p\leq N',\ p\not\in M,\ -L\leq\ell_{p_1},\ldots,\ell_{p_{|M|}}\leq L\right\}\]
and 
\begin{align*}
x&:=x(p,\ell_{p_1},\ldots,\ell_{p_{|M|}})=-\left(d_k\log p+\sum_{j=1}^{|M|}\ell_{p_j}\log p_j \right),\\ 
y&:=y(q,\ell'_{p_1},\ldots,\ell'_{p_{|M|}})=-\left(d_k\log q+\sum_{j=1}^{|M|}\ell'_{p_j}\log p_j \right).
\end{align*}
We need to estimate
\[
\delta_k = \min_{\substack{x,y\in\Omega_k\\(p,\ell_{p_1},\ldots,\ell_{p_{|M|}})\ne (q,\ell'_{p_1},\ldots,\ell'_{p_{|M|})}}}|x-y|,
\]	
so we consider the following three cases.

\textit{Case 1:} $p=q$. Then $|x-y|>0$ by the fact that the sequence of logarithms of prime numbers are linearly independent over $\mathbb{Q}$.

\textit{Case 2:} $p\ne q$, $0\ne d_k\in\mathbb{Q}$. Similarly to the Case 1, by the unique factorization of integers, we have $x=y\leftrightarrow d_k=0,\ \ell'_{p_i}=\ell'_{p_i}$, so $|x-y|>0$.
	
\textit{Case 3:} $p\ne q$, $d_k\notin\mathbb{Q}$. Then $x=y$ means that for suitable $A,B\in\mathbb{Q}$ we have

\begin{equation}\label{eq:SixExpo}
d_k\log B = d_k(\log p-\log q) = \sum_{i=1}^{|M|}(\ell'_{p_i}-\ell_{p_i})\log p_i = \log A.
\end{equation}

Using the Six Exponentials Theorem one can prove (see \cite[Corollary 2.3 and Lemma 2.4]{PanWuerzburg}) the following result.
\begin{lemma}
For arbitrary irrational number $d$ there exist at most two primes $p_1,p_2$  such that if
\[
d = \frac{\log a}{\log b},\qquad\text{for some $a,b\in\mathbb{Q}$},
\]
then at least one of $p_1,p_2$ appears in the prime factorization of $b$.
\end{lemma}
Now, let us note that $B=p/q$ for some $p,q\notin M$, so $p,q>y$. Hence, taking sufficiently large~$y$ and using the above lemma we obtain that \eqref{eq:SixExpo} never holds, which means that $|x-y|>0$ in Case 3 as well.

Summing up all three cases we get $\delta_k>0$. Hence using again a suitable version of Hilbert's inequality due to Montgomery and Vaughan (see \cite[Corollary 2, Eq. (1.8)]{MV}) we see that for some $\theta_k$ with $|\theta_k|\leq 1$ and sufficiently large $T$ we have
\begin{align*}
\int_T^{2T}|\tilde{P}(\tau)|^2|S_1|^2d\tau &= (T+2\pi\delta_k^{-1}\theta_k)\sum_{-L\leq \ell_{p_1},\ldots,\ell_{p_{|M|}}\leq L}|\tilde{c}(\ell_{p_1},\ldots,\ell_{p_{|M|}})|^2\sum_{\substack{p\leq N'\\p\not\in M}}p^{-2\sigma}\\
&\ll Ty^{1-2\sigma}\sum_{-L\leq \ell_{p_1},\ldots,\ell_{p_{|M|}}\leq L}|\tilde{c}(\ell_{p_1},\ldots,\ell_{p_{|M|}})|^2.
\end{align*}
Similarly one can obtain
\[
\int_T^{2T}|\tilde{P}(\tau)|^2d\tau = (T+2\pi\delta_0^{-1}\theta_0)\tilde{c}(\ell_{p_1},\ldots,\ell_{p_{|M|}})|^2
\]
for some positive real $\delta_0$ and $theta_0$ with $|\theta_0|\leq 1$. Thus, we have
\begin{align*}
T|\tilde{c}(\ell_{p_1},\ldots,\ell_{p_{|M|}})|^2\ll \int_T^{2T}|\tilde{P}(\tau)|^2d\tau&\ll\int_T^{2T}\tilde{\xi}(\tau)d\tau+T2^{2|M|-1}|M|^2\varepsilon_0^2\\
&\leq \left|\int_T^{2T}\tilde{P}(\tau)\right|+T(2^{|M|-1}|M|\varepsilon_0+2^{2|M|-1}|M|^2\varepsilon_0^2).
\end{align*}

Note that from \eqref{eq:Ptilde} we get
\[
\int_T^{2T}\tilde{P}(\tau)d\tau = c(0)^{|M|}T+O(1).
\]
Moreover, we have
\[
|c(0)| = \left|\int_{-1/2}^{1/2}P(x)dx\right|\leq \int_{-1/2}^{1/2}(|\xi(x)|+\varepsilon_0)dx = 2d+\delta'+\varepsilon_0.
\]
Therefore, 
\[
\int_T^{2T}|\tilde{P}(\tau)|^2|S_1|^2d\tau\ll Ty^{1-2\sigma}\left((2d+\delta'+\varepsilon_0)^{|M|}+2^{|M|-1}|M|\varepsilon_0+2^{2|M|-1}|M|^2\varepsilon_0^2\right).
\]
Similarly using  Hilbert's inequality one can easily obtain that
\[
\int_T^{2T}|S_1|^2d\tau\ll Ty^{1-2\sigma},
\]
so 
\[
I_1\ll Ty^{1-2\sigma}\left((2d+\delta'+\varepsilon_0)^{|M|}+2^{2|M|+1}|M|^2\varepsilon_0\right).
\]

Next, in order to apply Hilbert's inequality to estimate $I_2$, we need to use the following estimate
\[
\delta_k(T) = d_k\min_{\substack{N'<p,q\leq T^\delta\\p,q\notin M}}|\log p-\log q|\geq d_k\min_{N'<p\leq T^\delta}\log\left(1+\frac{1}{p-1}\right)\geq \frac{d_k}{T^\delta}.
\]
So
\[
I_2\leq \left(T+\frac{2\pi T^\delta}{d_k}\right)\sum_{\substack{N'<p\leq T^\delta\\p\not\in M}}\frac{1}{p^{2\sigma}}\ll TN'^{1-2\sigma}.
\]
Choosing $\delta'$ and $\varepsilon_0$ such that
\[
\left(1+\frac{\delta'+\varepsilon_0}{2d}\right)^{|M|}<\frac{3}{2}\qquad\text{and}\qquad 2^{2|M|+1}|M|^2\varepsilon_0<\frac{1}{2}(2d)^{|M|}
\]
implies that $I_1\ll Ty^{1-2\sigma}(2d)^{|M|}$. 

Moreover, if $N'$ is such large that $N'^{1-2\sigma}\leq (2d)^{|M|}y^{1-2\sigma}$, then
\[
\int_{A_d(T)}\left|\sum_{\substack{p\leq T^\delta\\p\not\in M}}\frac{1}{p^{s+id_k\tau}}\right|^2d\tau\ll T(2d)^{|M|}y^{1-2\sigma},
\]
which completes the proof.
\end{proof}

Now, we are in the position to complete the proof of our theorem.

\begin{proof}[Proof of Theorem \ref{thm:main}]
From Lemma \ref{lem:denseness} we know that there exist a finite set of primes $M\supset \{p:p<y\}$ and a sequence $(\theta_p)_{p\in M}$ of real numbers such that
\[
\max_{1\leq k\leq n}\left|\sum_{p\in M}\frac{e(d_k\theta_p)}{p^s}-z_k\right|<\varepsilon.
\]
Let $A_d(T) = \{\tau\in [T,2T]: \max_{p\in M}\|-\frac{\tau \log p}{2\pi} - \theta_p\|\leq d\}$ and $d$ is such small that
\[
\max_{1\leq k\leq n}\left|\sum_{p\in M}\frac{e(-d_k\frac{\tau\log p}{2\pi})}{p^s}-\sum_{p\in M}\frac{e(d_k\theta_p)}{p^s}\right|<\varepsilon.
\]
Let us emphasize that the choice of such $d$ is possible since $d_1,\ldots,d_n\in\mathbb{N}$, and it is the only step in the proof where we need to assume that $d_1,\ldots,d_n$ are integers. 

Let us notice that for $\tau\in A_d(T)$ we have
\[
\max_{1\leq k\leq n}\left|\sum_{p\in M}\frac{1}{p^{s+id_k\tau}}-z_k\right|<2\varepsilon.
\]

Next let
\[
B(T,\varepsilon) = \left\{\tau\in A_d(T): \max_{1\leq k\leq n}\left|\sum_{\substack{p\leq T^\delta\\p\not\in M}}\frac{1}{p^{s+id_k\tau}}\right|<\varepsilon\right\},
\]
where $0<\delta<\frac{1}{65}$. Then, by Lemma~\ref{lem:middle}, for sufficiently large $y$ we have 
\[\frac{1}{T}\meas (B(T,\varepsilon))>\frac{1}{2}(2d)^{|M|}.\]
Thus, using the notation from Lemma~\ref{lem:approxFiniteSum}, we get that 
\[
\max_{1\leq k\leq n}|\log\zeta(s+id_k\tau)-z_k|<4\varepsilon\qquad \text{for $\tau\in B(T,\varepsilon)\cap C(T,\varepsilon)$,}
\]
and, since $\meas C(T,\varepsilon)>T(1-\varepsilon')$ for every $\varepsilon'>0$, we have
\[
\frac{1}{T}\meas B(T,\varepsilon)\cap C(T,\varepsilon)>\frac{1}{T}\meas B(T,\varepsilon)+\frac{1}{T}\meas C(T,\varepsilon) - 1>\frac{1}{2}(2d)^{|M|}-\varepsilon'
\]
which completes the proof, since the right hand side is positive for sufficiently small $\varepsilon'$.
\end{proof}


\begin{thebibliography}{10}
	
\bibitem {B81} B. Bagchi, \textit{The statistical Behavior and Universality Properties of the Riemann Zeta-Function and Other Allied Dirichlet Series}, Ph.D. Thesis, Calcutta, Indian Statistical Institute, (1981).
	
\bibitem{B87} {B. Bagchi}, {\it Recurrence in topological dynamics and the Riemann hypothesis}, Acta Math. Hung. {\bf 50} (1987), 227--240.

\bibitem{Bohr} H. Bohr, \textit{Z\"ur Theorie der Riemann'schen Zetafunktion im kritischen Streifen}, Acta Math. \textbf{40} (1915), 67--100.

\bibitem{BohrCourant} H. Bohr, R. Courant, \textit{Neue Anwendungen der Theorie der Diophantischen Approximationen auf die Riemannschen Zetafunktion}, J. reine Angew. Math. \textbf{144} (1914), 249--274.

\bibitem{BohrJessen} H. Bohr, B. Jessen, \textit{\"Uber die Werteverteilung der Riemannschen Zetafunktion}, Erste Mitteilung, Acta Math. \textbf{54} (1930), 1--35; Zweite Mitteilung, ibid. \textbf{58} (1932), 1--55.
	
\bibitem{chen} Y.-G. Chen, \textit{The best quantitative Kronecker's theorem}, J. Lond. Math. Soc. (2) \textbf{61} (2000), no. 3, 691--705.

\bibitem{Dickson} D.G. Dickson, \textit{Zeros of exponential sums}, Proc. Amer. Math. Soc. \textbf{16} (1965), no. 1, 84--89.

\bibitem {GottschalkHedlund} W.H. Gottschalk, G.A. Hedlund, M. Kulas, \textit{Recursive properties of topological transformation groups}, Bull. Amer. Math. Soc. \textbf{52} (1946), 488-489.

\bibitem{MV} H.L. Montgomery and R.C. Vaughan, \textit{Hilbert's inequality}, J. London Math. Soc. (2) \textbf{8} (1974), 73--82.

\bibitem{N09} T. Nakamura, \textit{The joint universality and the generalized strong recurrence for Dirichlet $L$-functions}, Acta Arith. {\bf 138} (2009), 357--362.

\bibitem{NP} {T. Nakamura, \L. Pa\'nkowski}, \textit{Erratum to: The generalized strong recurrence for non-zero rationals parameters}, Arch. Math. \textbf{99} (2012), 43--47.

\bibitem{PanWuerzburg} \L. Pa\'{n}kowski, \textit{Some remarks on the generalized strong recurrence for L-functions}, in: New Directions in Value Distribution Theory of Zeta and L-Functions, in: Ber. Math., Shaker Verlag, Aachen, 2009, pp.305--315.

\bibitem{P2019} \L. Pa\'{n}kowski, \textit{Joint universality and generalized strong recurrence for the Riemann zeta function with rational parameter}, J. Number Theory \textbf{163} (2016), 61--74.
 
\bibitem{Pe} D.V. Pecherskii, \textit{On rearrangements of terms in functional series}, Soviet Math. Dokl. \textbf{14} (1973), 633--636.

\bibitem{P09} {\L. Pa\'nkowski}, {\it Some remarks on the generalized strong recurrence for $L$-functions},
in: New Directions in Value Distribution Theory of zeta and L-Functions, Ber.Math., Shaker Verlag, Aachen,(2009), 305--315.

\bibitem{Tsang} K.M. Tsang, \textit{The distribution of the values of the Riemann zeta-function}, Ph.D. thesis. Princeton University, Princeton (1984).

\bibitem{V} S.M. Voronin, \textit{Theorem on the universality of the Riemann zeta function}, Izv. Akad. Nauk SSSR Ser. Mat. \textbf{39} (1975), 475--486 (in Russian); Math. USSR Izv. \textbf{9} (1975), 443--453.

\end{thebibliography}
\end{document}